\documentclass[a4paper]{amsart}
\usepackage{graphicx}
\usepackage{amssymb}
\usepackage{amsmath}
\usepackage{amsthm}
\usepackage{amscd}
\usepackage{color}
\usepackage{colordvi}
\usepackage[all,2cell]{xy}
\usepackage{tikz}
\usepackage{picture}
\usepackage{multirow}
\usepackage{makecell}

\UseAllTwocells \SilentMatrices
\newtheorem{thm}{Theorem}[section]

\newtheorem{cor}[thm]{Corollary}
\newtheorem{lem}[thm]{Lemma}
\newtheorem{exm}[thm]{Example}

\newtheorem{prop}[thm]{Proposition}
\theoremstyle{definition}
\newtheorem{defn}[thm]{Definition}
\theoremstyle{remark}
\newtheorem{rem}[thm]{\bf Remark}
\numberwithin{equation}{section}

\begin{document}
\title[The finite EI categories of Cartan type]{The finite EI categories of Cartan type}
\author[Xiao-Wu Chen, Ren Wang] {Xiao-Wu Chen, Ren Wang$^*$}

\thanks{$^*$ The corresponding author}
\thanks{}
\date{\today}

\thanks{E-mail: xwchen$\symbol{64}$mail.ustc.edu.cn; renw$\symbol{64}$mail.ustc.edu.cn}

\maketitle

\dedicatory{}%
\commby{}%

\begin{abstract}
To each  symmetrizable Cartan matrix, we associate a finite free EI category. We prove that the corresponding category algebra is isomorphic to the algebra defined in [C. Geiss, B. Leclerc, and J. Schr\"{o}er, {\em Quivers with relations for symmetrizable Cartan matrices I: Foundations},  Invent. Math. {\bf 209} (2017), 61--158], which is associated to another symmetrizable Cartan matrix. In certain cases, the  algebra isomorphism provides an algebraic enrichment of the well-known correspondence between symmetrizable Cartan matrices and graphs with automorphisms.
\end{abstract}

\section{Introduction}

Recall that a finite category $\mathcal{C}$ is EI, provided that each \underline{e}ndomorphism is an \underline{i}somorphism. Finite EI categories arise naturally in the modular representation theory of finite groups; see \cite{Lin04, We08, Lin14}. The notion of a free EI category is introduced in \cite{L2011}. Each finite EI category is a quotient of a free EI category.

For each symmetrizable Cartan matrix,  there is a class of finite dimensional algebras  \cite{GLS},  whose modules categorify the corresponding root system over an arbitrary field $k$.  More precisely, let $(C, D, \Omega)$ be a Cartan triple, where $C$ is a symmetrizable  Cartan matrix, $D$ is its symmetrizer and $\Omega$ is an acyclic orientation of $C$. A $k$-algebra $H(C, D, \Omega)$ is introduced in \cite{GLS}, where the rank vectors of certain $H(C, D, \Omega)$-modules are some positive roots of the Lie algebra associated to $C$.

There are remarkable common features of the work \cite{L2011} and \cite{GLS}: the category algebra $k\mathcal{C}$ of a free EI category $\mathcal{C}$ is a tensor algebra, and is $1$-Gorenstein under mild assumptions \cite{Wang}; the algebra $H(C, D, \Omega)$  is also a tensor algebra,  and is always $1$-Gorenstein; see also \cite{Geu}; moreover, both algebras are closely related to the path algebras of finite acyclic quivers. We mention the work \cite{Ku}, which provides a categorical perspective for the algebra $H(C, D, \Omega)$,  and the  paper \cite{CL}, inspired by \cite{GLS},  where Gorenstein tensor algebras are studied in general.

In view of these common features, the following question arises: when are the above algebras $k\mathcal{C}$ and $H(C, D, \Omega)$ isomorphic? Another reason to be interested in this question is that it might establish new links between the modular representation theory and Lie theory.

We give a partial answer to the above question. Indeed, inspired by \cite{GLS}, we associate to each Cartan triple $(C, D, \Omega)$  a free EI category $\mathcal{C}(C, D, \Omega)$, called a \emph{free EI category of Cartan type}. We prove that the category algebra $k\mathcal{C}(C, D, \Omega)$ is isomorphic to the algebra $H(C', D', \Omega')$ associated to another, usually different,  Cartan triple $(C', D', \Omega')$; see Theorem~\ref{thm:main}.

The construction of $(C', D', \Omega')$ from $(C, D, \Omega)$ depends on the characteristic $p$ of the base field $k$; see Subsection \ref{sub:4.2}. The case where $p$ is zero or  coprime to each entry of $D$ seems to be of particular interest; in this case $C'$ is symmetric and $D'$ is the identity matrix. Then the algebra $H(C', D', \Omega')$ is isomorphic to the path algebra of an acyclic quiver of type $C'$. There is a well-known correspondence between symmetrizable Cartan matrices and graphs with automorphisms in \cite[Section 14.1]{Lus}. It turns out that our construction is compatible with this correspondence. More precisely, if $p$ is zero or coprime to each entry of $D$, the Cartan matrix $C$ corresponds exactly to the graph of $C'$ with a certain automorphism by the mentioned correspondence. Roughly speaking, our algebra isomorphism between the category algebra $k\mathcal{C}(C, D, \Omega)$ of type $C$ and the path algebra  of type $C'$ might be viewed as an algebraic enrichment of the correspondence in \cite{Lus}; see Proposition~\ref{prop:compa}. We mention that a different enrichment appears implicitly in \cite[Section~6]{Hu04} using species and  Galois extensions of the base field.

The paper is structured as follows. In Section 2, we recall basic facts on free EI categories. We recall from \cite{GLS} the algebra $H(C, D, \Omega)$ in Section 3. We introduce the notion of a free EI category of Cartan type in Section 4. The construction of  $(C', D', \Omega')$ from $(C, D, \Omega)$  is given in Subsection \ref{sub:4.2}. The core in the proof of  Theorem~\ref{thm:main}  is Proposition~\ref{prop:main} in Section 5, which states an explicit decomposition of a tensor bimodule.

In Subsection 6.1, we prove that in the construction of Subsection~\ref{sub:4.2}, the valued graph of $C$ is a disjoint union of Dynkin graphs (\emph{resp}. Euclidean graphs, graphs of indefinite type) if and only if so is the valued graph of $C'$; see Proposition~\ref{prop:type}.  In Subsection 6.2, we compare the construction with the correspondence in \cite[Section 14.1]{Lus}.

\section{Finite free EI categories}

 In this section,  we recall the construction of a finite free EI category from a finite EI quiver. We observe that the category algebra of the resulted EI category is isomorphic to a certain tensor algebra.

Let $\mathcal{C}$ be a finite category, that is, a category with only finitely many morphisms. Consequently, the category $\mathcal{C}$ contains only finitely many objects. Denote by ${\rm Mor}(\mathcal{C})$ (resp. ${\rm Obj}(\mathcal{C})$) the finite set of morphisms (resp. objects) in $\mathcal{C}$. The finite category $\mathcal{C}$ is said to be EI provided that every endomorphism is an isomorphism. Therefore, for each object $x$, ${\rm Hom}_{\mathcal{C}}(x,x)={\rm Aut}(x)$ is a finite group.

Let $k$ be a field. The \emph{category algebra} $k\mathcal{C}$ of $\mathcal{C}$ over $k$ is defined as
follows: $k\mathcal{C}=\bigoplus\limits_{\alpha \in {\rm Mor}(\mathcal{C})} k\alpha$ as
a $k$-vector space and the product  is given by the rule
\[\alpha  \beta=\left\{\begin{array}{ll}
\alpha\circ\beta, & \text{ if }\text{$\alpha$ and $\beta$ can be composed in $\mathcal{C}$}; \\
0, & \text{otherwise.}
\end{array}\right.\]
The unit is given by $1_{k\mathcal{C}}=\sum\limits_{x \in {\rm Obj}(\mathcal{C})}{\rm Id}_x$,
where ${\rm Id}_x$ denotes the identity endomorphism of $x$.

Let $G$ and $H$ be two groups. Recall that a \emph{$(G,H)$-biset} is a  set $X$ with a left $G$-action and a right $H$-action such that the two actions commute, that is,  $(gx)h=g(xh)$ for all $g\in G, x\in X$ and $h\in H$.

Let $X$ be a $(G,H)$-biset and $Y$ be a $(H,K)$-biset. The \emph{biset product} of $X$ and $Y$, denoted by $X\times_H Y$, is the set $X\times Y/\sim$ of equivalence classes under the equivalence relation $(x,hy)\sim (xh,y)$ for $x\in X, h\in H$ and $y\in Y$. By abuse of notation, the elements in $X\times Y/\sim$ are still denoted by $(x,y)$ for $x\in X$ and $y\in Y$. The set $X\times_H Y$ is naturally a $(G,K)$-biset; see \cite[Section 8]{We00}.

For a $(G,H)$-biset $X$, we denote by $kX$ the $k$-vector space with a $k$-basis $X$. Then $kX$ is naturally a $kG$-$kH$-bimodule.

The following fact is well known.

\begin{lem}\label{lem:bi-prod}
Let $X$ be a $(G,H)$-biset and $Y$ be a $(H,K)$-biset. Then there is an isomorphism of $kG$-$kK$-bimodules
\[ kX\otimes_{kH} kY\overset{\sim}{\rightarrow} k(X\times_H Y), \ \ \ \ x\otimes y\mapsto (x,y), \]
for $x\in X$ and $y\in Y$. \hfill $\square$
\end{lem}

Recall that a finite quiver $Q=(Q_0, Q_1; s,t)$ consists of a finite set $Q_0$ of vertices, a finite set $Q_1$ of arrows endowed with two maps $s, t\colon Q_1\rightarrow Q_0$.  For an arrow $\alpha$, $s(\alpha)$ and $t(\alpha)$ are called its starting vertex and terminating vertex, respectively. Then the arrow $\alpha$ is visualized as $\alpha\colon s(\alpha)\rightarrow t(\alpha)$.  For $n\geq 2$, a path $p=\alpha_n\cdots \alpha_2\alpha_1$ of length $n$ consists of arrows $\alpha_i$ such that $t(\alpha_i)=s(\alpha_{i+1})$ for $1\leq i<n$; we set $s(p)=s(\alpha_1)$ and $t(p)=t(\alpha_n)$. Here, we write the concatenation from the right to the left. We denote by $Q_n$ the set of paths of length $n$. We observe that a path of length one is just an arrow. We identify each vertex $i$ with the corresponding path  $e_i$ of length zero.

Denote by $kQ=\bigoplus_{n\geq 0} kQ_n$ the path algebra of $Q$, whose multiplication is given by the concatenation of paths. In particular, each path $e_i$ of length zero is an idempotent. Then we have  $1_{kQ}=\sum_{i\in Q_0} e_i$.   The quiver $Q$ is \emph{acyclic} provided that it does not contain oriented cycles, in which case the path algebra $kQ$ is finite dimensional.

 Recall from \cite[Definition 2.1]{L2011} that a \emph{finite EI quiver} $(Q, X)$ consists of a finite acyclic quiver $Q$ and an assignment $X=(X(i), X(\alpha))_{i\in Q_0, \alpha\in Q_1}$. More precisely, for each vertex $i\in Q_0$, $X(i)$ is a finite group, and for each arrow $\alpha$, $X(\alpha)$ is a finite $(X(t(\alpha)), X(s(\alpha)))$-biset. For a path $p=\alpha_n\cdots \alpha_2\alpha_1$ of length $n$, it is natural to define
 $$X(p)=X({\alpha_n}) \times_{X({t(\alpha_{n-1})})} X({\alpha_{n-1}}) \times_{X({t(\alpha_{n-2})})} \cdots \times_{X({t(\alpha_2)})} X({\alpha_2})\times_{X({t(\alpha_1)})} X({\alpha_1})
 .$$
  Then $X(p)$ is naturally a $(X(t(p)), X(s(p)))$-biset. Indeed, for two paths $p, q$ satisfying $s(p)=t(q)$, we have a natural isomorphism
  \begin{align}\label{iso:1}
  X(p)\times_{X(t(q))} X(q)\stackrel{\sim}\longrightarrow X(pq),
  \end{align}
  where $pq$ denotes the concatenation. We identify $X(e_i)$ with $X(i)$.

Each finite EI quiver  $(Q, X)$ gives  rise to a finite EI category $\mathcal{C}=\mathcal{C}(Q, X)$. The objects of $\mathcal{C}$ coincide with the vertices of $Q$. For two objects $x$ and $y$, we have
$${\rm Hom}_\mathcal{C}(x, y)=\bigsqcup_{\{p {\rm{ \; paths \; in\;  }} Q {\rm{\; with \; }} s(p)=x  {\rm{\; and \; }} t(p)=y\}} X(p).$$
The composition is induced by the concatenation of paths and the isomorphism (\ref{iso:1}). Since $Q$ is acyclic, we infer that $\mathcal{C}$ is a finite category; moreover, we have that ${\rm Hom}_\mathcal{C}(x, x)=X(x)$ is a finite group. It follows that the category  $\mathcal{C}$ is indeed finite EI.

Following \cite[Definition 2.2]{L2011}, we call the above category $\mathcal{C}(Q, X)$ the \emph{free EI category} associated to the finite EI quiver $(Q, X)$. More generally, a finite EI category $\mathcal{C}$ is called \emph{free} provided that it is isomorphic to $\mathcal{C}(Q, X)$ for some finite EI quiver $(Q, X)$.  For alternative characterizations of finite free EI categories, we refer to \cite[Proposition 2.8]{L2011}, \cite[Proposition 4.5]{Wang} and \cite[Proposition 2.1]{Wang17}.

We have the following observation.

\begin{prop}\label{prop:iso}
Let $\mathcal{C}=\mathcal{C}(Q, X)$ be the finite free EI category associated to a finite EI quiver $(Q, X)$. Set $A=\prod_{x\in Q_0} kX(x)$ and $V=\bigoplus_{\alpha\in Q_1} kX(\alpha)$, where $V$ is naturally an $A$-$A$-bimodule. Then there is an isomorphism $T_A(V)\simeq k\mathcal{C}$ of algebras.
\end{prop}

Here, $T_A(V)=A\oplus V\oplus V^{\otimes 2} \oplus \cdots \oplus V^{\otimes n} \oplus \cdots $ denotes the tensor algebra, where $V^{\otimes n}=V\otimes_A\cdots \otimes_A V$ is the $n$-fold tensor product. The tensor algebra $T_A(V)$ is naturally $\mathbb{N}$-graded.

\begin{proof}
We observe that $k\mathcal{C}=\bigoplus_{n\geq 0} (k\mathcal{C})_n$ is naturally $\mathbb{N}$-graded, where $(k\mathcal{C})_n=k(\bigsqcup_{p\in Q_n}X(p))$. Hence, we have $(k\mathcal{C})_0=A$ and $(k\mathcal{C})_1=V$. By the universal property of tensor algebras, there is a unique  homomorphism $\phi\colon T_A(V)\rightarrow k\mathcal{C}$ between $\mathbb{N}$-graded algebras such that its restriction to $A\oplus V$ is the identity.

It suffices to prove that $\phi$ induces an isomorphism $V^{\otimes n}\simeq (k\mathcal{C})_n$ for each $n\geq 2$. For this, we observe  $V=\bigoplus_{x, y\in Q_0} {_yV_x}$, where
$$_yV_x=k(\bigsqcup_{\{\alpha\in Q_1\; |\; s(\alpha)=x, t(\alpha)=y\}} X(\alpha))$$
 is naturally a $kX(y)$-$kX(x)$-bimodule. We have the following isomorphisms
\begin{align*}
& V^{\otimes n} \simeq \bigoplus_{x_0, x_1, \cdots, x_n\in Q_0} {_{x_n}V_{x_{n-1}}} \otimes_{kX(x_{n-1})} \cdots \otimes_{kX(x_2)} {_{x_2}V_{x_1}} \otimes_{kX(x_1)} {_{x_1}V_{x_0}}\\
&\simeq \bigoplus_{\{\alpha_i\in Q_1\; |\; t(\alpha_i)=s(\alpha_{i+1})\}} k(X(\alpha_n) \times_{X(t(\alpha_{n-1}))} \cdots \times_{X(t(\alpha_2))} X(\alpha_2)\times_{X(t(\alpha_1))} X(\alpha_1))\\
&=\bigoplus_{p\in Q_n} kX(p)\simeq (k\mathcal{C})_n,
\end{align*}
where  the second isomorphism follows by applying Lemma \ref{lem:bi-prod} repeatedly. Since the restriction of $\phi$ on $V^{\otimes n}$ coincides with this composite isomorphism, we are done.
\end{proof}

\begin{rem}
Assume that the above assignment $X$ is trivial, that is, all sets $X(x)$ and $X(\alpha)$ are consisting of single elements. Then the category algebra $k\mathcal{C}$ coincides with the path algebra $kQ$ of the quiver $Q$. The isomorphism in Proposition~\ref{prop:iso} in this case is well known; see \cite[Proposition III.1.3]{ARS}.
\end{rem}

\section{The algebras associated to Cartan matrices}

In this section, we recall the algebras defined in \cite{GLS}, which are associated to symmetrizable generalized  Cartan matrices. For two nonzero integers $l,m \in \mathbb{Z}$, we denote by ${\rm gcd}(l,m)$ their greatest common divisor, which is always assumed to be positive.

Let $n\geq 1$ be a positive integer. A matrix $C=(c_{ij})\in M_n(\mathbb{Z})$ is called a \emph{symmetrizable generalized Cartan matrix} provided that the following conditions hold:
\begin{enumerate}
	\item[(C1)] $c_{ii}=2$ for all $i$;
	\item[(C2)] $c_{ij}\leq 0$ for all $i\neq j$, and $c_{ij}<0$ if and only if $c_{ji}<0$;
	\item[(C3)] There is a diagonal  matrix $D={\rm diag}(c_1,\cdots,c_n)\in M_n(\mathbb{Z})$ with $c_i\geq 1$ for all $i$ such that the product matrix $DC$ is symmetric.
\end{enumerate}

The matrix $D$ appearing in (C3) is called a \emph{symmetrizer} of $C$. From now on, by a \emph{Cartan matrix} we always mean a symmetrizable generalized Cartan matrix.

Let $C=(c_{ij})\in M_n(\mathbb{Z})$ be a Cartan matrix. An \emph{orientation} of $C$ is a subset $\Omega\subset \{1,2,\cdots,n\}\times \{1,2,\cdots,n\}$ such that the following conditions hold:
\begin{enumerate}
	\item $\{(i,j),(j,i)\}\cap \Omega\neq \emptyset$ if and only if $c_{ij}<0$;
	\item For each sequence ($(i_1,i_2),(i_2,i_3),\cdots,(i_t,i_{t+1})$) with $t\geq 1$ and $(i_s,i_{s+1})\in \Omega$ for all $1\leq s\leq t$,  we have $i_1\neq i_{t+1}$.
\end{enumerate}
For an orientation $\Omega$ of $C$,  let $Q=Q(C,\Omega)$ be the finite quiver with the set of vertices $Q_0:=\{1,2, \cdots,n\}$ and with the set of arrows
\[Q_1:=\{\alpha^{(g)}_{ij}:j\rightarrow i\mid(i,j)\in \Omega, 1\leq g\leq {\rm gcd}(c_{ij},c_{ji})\}\cup\{\varepsilon_i:i\rightarrow i\mid 1\leq i\leq n\}.\]
 We call $Q$ a \emph{quiver of type $C$}. Let $Q^{\circ}=Q^{\circ}(C,\Omega)$ be the quiver obtained from $Q$ by deleting all loops $\varepsilon_i$. Then $Q^{\circ}$ is a finite acyclic quiver.

 We will call $(C, D, \Omega)$ a \emph{Cartan triple}, where $C$ is a Cartan matrix, $D$ its symmetrizer and $\Omega$ an orientation of $C$.

\begin{defn}{\rm (\cite[Section 1.4]{GLS})}\label{defn:GLS}
Let $k$ be a field, and $(C, D, \Omega)$ be a Cartan triple.  Then we have the quiver $Q=Q(C,\Omega)$. Let
\[H=H(C,D,\Omega)=kQ/I,\]
where $kQ$ is the path algebra of $Q$, and $I$ is the two-sided ideal of $kQ$ defined by the following relations:
\begin{enumerate}
	\item[(H1)] For each vertex $i$, we have the \emph{nilpotency relation}
	\[\varepsilon_i^{c_i}=0.\]
	\item[(H2)] For each $(i,j)\in \Omega$ and each $1\leq g\leq {\rm gcd}(c_{ij},c_{ji})$, we have the \emph{commutativity relation}
	\[\varepsilon_i^{\frac{c_i}{{\rm gcd}(c_i,c_j)}} \alpha^{(g)}_{ij}=\alpha^{(g)}_{ij}\varepsilon_j^{\frac{c_j}{{\rm gcd}(c_i,c_j)}}.\]
\end{enumerate}
\end{defn}

\begin{rem}\label{rem:GLS}
Assume that the Cartan matrix $C$ is symmetric. We take $D=cI_n$ to be a scalar matrix for $c\geq 1$. Then the algebra $H(C, D, \Omega)$ is isomorphic to $kQ^\circ\otimes k[\varepsilon]/(\varepsilon^c)$. This algebra is studied in \cite{Fan, RZ} from different perspectives. The main concern of \cite{GLS} is non-symmetric Cartan matrices.
\end{rem}

Let us introduce some notation. Fix  the Cartan triple $(C, D, \Omega)$  as above. For each $1\leq i\leq n$, set $H_i=k[\epsilon_i]/(\epsilon_i^{c_i})$. For each $(i,j)\in \Omega$, set $H_{ij}=k[\epsilon_{ij}]/(\epsilon_{ij}^{{\rm gcd}(c_i,c_j)})$. Consider the following injective algebra homomorphisms
\[H_{ij}\rightarrow H_i,\ \ \ \epsilon_{ij}\mapsto \epsilon_i^{\frac{c_i}{{\rm gcd}(c_i,c_j)}}\]
and
\[H_{ij}\rightarrow H_j,\ \ \epsilon_{ij}\mapsto \epsilon_j^{\frac{c_j}{{\rm gcd}(c_i,c_j)}}.\]
For each $1\leq g\leq {\rm gcd}(c_{ij},c_{ji})$, set $_iH_j^{(g)}:= H_i\otimes_{H_{ij}} H_j$, which is naturally an $H_i$-$H_j$-bimodule. The element $1\otimes 1$ in $_iH_j^{(g)}$ will be denoted by $\beta^{(g)}_{ij}$.
Set $B=\prod\limits_{i=1}^{n} H_i$ and $W=\bigoplus_{(i,j)\in \Omega}  \bigoplus_{g=1}^{{\rm gcd}(c_{ij},c_{ji})} {_iH_j^{(g)}}$. Then $W$ is naturally a $B$-$B$-bimodule.

The following result is essentially due to \cite[Proposition 6.4]{GLS}. We make slight modification and give a complete proof. Recall the notation $H=H(C, D, \Omega)$.

\begin{prop}\label{prop:H}
Keep the notation as above. Then there is an isomorphism of algebras
\[T_B(W)\stackrel{\sim}\longrightarrow H,\]
which induces an isomorphism of algebras
\[H_i\simeq e_iHe_i\]
for each $1\leq i\leq n$, and an isomorphism of $H_i$-$H_j$-bimodules
\[{_iH_j^{(g)}}\simeq {\rm Span}_k\{ \varepsilon_i^{n_i}\alpha^{(g)}_{ij}\varepsilon^{n_j}_j\ \mid\  0\leq n_i\leq c_i-1,0\leq n_j\leq c_j-1 \}\]
for each $(i,j)\in \Omega$ and $1\leq g\leq {\rm gcd}(c_{ij},c_{ji})$.
\end{prop}

Here, $e_i$ denotes the idempotent of $H$ corresponding to the vertex $i$. The notation ``Span$_k$" means the subspace of $H$ spanned by the mentioned elements.

\begin{proof}
There is an algebra homomorphism $\phi_0\colon B\rightarrow H$ sending $\epsilon_i$ to $\varepsilon_i$. Moreover, each $(i,j)\in \Omega$ and $1\leq g\leq {\rm gcd}(c_{ij},c_{ji})$,  there is a bilinear map  $H_i\times H_j \rightarrow e_iHe_j$ sending $(\epsilon_i^a, \epsilon_j^b)$ to $\varepsilon_i^a\alpha_{ij}^{(g)}\varepsilon_j^b$. By the commutativity relations in $H$, this map is $H_{ij}$-balanced. In particular, we have an induced $H_i$-$H_j$-bimodule homomorphism $_iH_j^{(g)}\rightarrow e_iHe_j$ sending $\beta_{ij}^{(g)}$ to $\alpha_{ij}^{(g)}$. These bimodule homomorphisms give rise to a $B$-$B$-bimodule homomorphism $\phi_1\colon W\rightarrow H$. By the universal property of tensor algebras, we have an algebra homomorphism $\phi\colon T_B(W)\rightarrow H$ extending $\phi_0$ and $\phi_1$.

On the other hand, there is a unique algebra homomorphism $\psi\colon H\rightarrow T_B(W)$ such that $\psi(e_i)=1_{H_i}$, $\psi(\varepsilon_i)=\epsilon_i$ and $\psi(\alpha^{(g)}_{ij})=\beta^{(g)}_{ij}$. It is direct to verify the defining relations of $H$ in $T_B(W)$, using the identity
$$\epsilon_i^{\frac{c_i}{{{\rm gcd}(c_i, c_j)}}}\beta_{ij}^{(g)}=\beta_{ij}^{(g)} \epsilon_j^{\frac{c_j}{{\rm gcd}(c_i, c_j)}}$$
 in $_iH_{j}^{(g)}$. Moreover, we have $\psi\circ \phi={\rm Id}_{T_B(W)}$ and $\phi\circ \psi={\rm Id}_H$, since they hold trivially on generators. Then we are done.
\end{proof}

\section{The category algebra of a free EI category of Cartan type}

In this section, we introduce a free EI category, which is associated to a Cartan triple. The main result states that the corresponding category algebra is isomorphic to the algebra associated to another Cartan triple; see Theorem \ref{thm:main}.

\subsection{Free EI categories of Cartan type}

Let $n\geq 1$ and $(C=(c_{ij}), D={\rm diag}(c_1, c_2, \cdots, c_n), \Omega)$ be a Cartan triple. Recall that $Q=Q(C, \Omega)$ is the quiver of type $C$, and that $Q^\circ$ is the quiver obtained from $Q$ by deleting all loops $\varepsilon_i$.

We will define a finite EI quiver $(Q^\circ, X)$. Recall that $Q^\circ_0=\{1, 2, \cdots, n\}$ and $Q^\circ_1=\{\alpha_{ij}^{(g)}\colon j\rightarrow i\; |\; (i, j)\in \Omega, 1\leq g\leq {\rm gcd}(c_{ij}, c_{ji})\}$. The assignment $X$ is given such that $X(i)=\langle \eta_i\; |\; \eta_i^{c_i}=1\rangle$ is a cyclic group of order $c_i$. For each $(i, j)\in \Omega$, we set $G_{ij}=\langle \eta_{ij}\; |\; \eta_{ij}^{{\rm gcd}(c_i, c_j)}=1\rangle$ to be a cyclic group of order ${\rm gcd}(c_i, c_j)$. There are injective group homomorphisms
\[G_{ij}\longrightarrow X(i), \;  \eta_{ij}\mapsto \eta_i^{\frac{c_i}{{\rm gcd}(c_i, c_j)}}\]
and
\[G_{ij}\longrightarrow X(j), \;  \eta_{ij}\mapsto \eta_j^{\frac{c_j}{{\rm gcd}(c_i, c_j)}}.\]
Then we have the $(X(i), X(j))$-biset $X(i)\times_{G_{ij}} X(j)$. We set
$$X(\alpha_{ij}^{(g)})=X(i)\times_{G_{ij}} X(j)$$
 for each $1\leq g\leq {\rm gcd}(c_{ij}, c_{ji})$. We observe that both the left $X(i)$-action and right $X(j)$-action on $X(\alpha_{ij}^{(g)})$ are free.

\begin{defn}\label{defn:EI}
For a Cartan triple $(C, D, \Omega)$, the free EI category $\mathcal{C}(Q^\circ, X)$ associated to the above finite EI quiver $(Q^\circ, X)$ will be denoted by $\mathcal{C}(C,D, \Omega)$, which will be called the \emph{free EI category of type} $C$.
\end{defn}

We give an intrinsic characterization of free EI categories of Cartan type.

Let $\mathcal{C}$ be a finite EI category. Recall that a non-isomorphism  $\alpha$ is called \emph{unfactorizable}, provided that $\alpha_1$ or $\alpha_2$ is an isomorphism in any factorization $\alpha=\alpha_1\circ \alpha_2$. Denote by ${\rm Hom}_\mathcal{C}^0(x, y)$ the set of unfactorizable morphisms between two objects $x$ and $y$; it is naturally an $({\rm Aut}(y), {\rm Aut}(x))$-biset. For example, if $\mathcal{C}=\mathcal{C}(C,D, \Omega)$, we have
$${\rm Hom}_\mathcal{C}^0(j, i)=\bigsqcup_{1\leq g\leq {\rm gcd}(c_{ij}, c_{ji})} X(\alpha_{ij}^{(g)}).$$
In particular, the category $\mathcal{C}(C,D, \Omega)$ satisfies the conditions (EC1)-(EC3) below.

\begin{prop}
Let $\mathcal{C}$ be a finite free EI category. Assume that the following conditions are satisfied:
\begin{enumerate}
\item[(EC1)] ${\rm Obj}(\mathcal{C})=\{x_1, x_2, \cdots, x_n\}$ and ${\rm Aut}(x_i)=\langle \eta_i\; |\; \eta_i^{c_i}=1\rangle$ is cyclic of order $c_i$ for each $i$;
\item[(EC2)] The left ${\rm Aut}(x_i)$-action and right ${\rm Aut}(x_j)$-action on ${\rm Hom}^0_\mathcal{C}(x_j, x_i)$ are both free for each $i, j$.
    \item[(EC3)] $\eta_i^{\frac{c_i}{{\rm gcd}(c_i,c_j)}} \circ \alpha=\alpha \circ \eta_j^{\frac{c_j}{{\rm gcd}(c_i,c_j)}}$ for each $\alpha\in {\rm Hom}^0_\mathcal{C}(x_j, x_i)$.
\end{enumerate}
Then $\mathcal{C}$ is isomorphic to $\mathcal{C}(C, D, \Omega)$ for some Cartan triple $(C, D, \Omega)$.
\end{prop}

\begin{proof}
We write the $({\rm Aut}(x_i), {\rm Aut}(x_j))$-biset ${\rm Hom}_\mathcal{C}^0(x_j, x_i)$ as a disjoint union of orbits
$${\rm Hom}_\mathcal{C}^0(x_j, x_i)=\bigsqcup_{l=1}^{m_{ij}} {\rm Aut}(x_i)\alpha_{ij}^{(l)} {\rm Aut}(x_j),$$
where $\{\alpha^{(l)}_{ij}\; |\; 1\leq l\leq m_{ij}\}$ is a complete representative set of orbits. Set $m_{ij}=0$ if ${\rm Hom}_\mathcal{C}^0(x_j, x_i)=\emptyset$.

We define a Cartan matrix $C=(c_{ij})\in M_n(\mathbb{Z})$ as follows: $c_{ii}=2$ for $1\leq i\leq n$; if both $m_{ij}$ and $m_{ji}$ are zero, we set $c_{ij}=0=c_{ji}$;  if $m_{ij}>0$, we set $c_{ij}= - \frac{c_j}{{\rm gcd}(c_i,c_j)}\cdot m_{ij}$ and  $c_{ji}= - \frac{c_i}{{\rm gcd}(c_i,c_j)}\cdot m_{ij}$, in which case we have $m_{ij}={\rm gcd}(c_{ij}, c_{ji})$.  Set $D={\rm diag}(c_1,c_2, \cdots,c_n)$, and $\Omega=\{(i,j)\; |\; {\rm Hom}_\mathcal{C}^0(x_j,x_i)\neq \emptyset \}$. Then $(C, D, \Omega)$ is a Cartan triple.

We claim that  $\mathcal{C}\simeq \mathcal{C}(C,D,\Omega)$. Recall from \cite[Proposition 2.9]{L2011} that free EI categories are completely determined by the automorphism groups of objects and the bisets of unfactorizable morphisms. It suffices to prove that for each $1\leq l\leq m_{ij}$,  the $({\rm Aut}(x_i), {\rm Aut}(x_j))$-bisets $X(\alpha_{ij}^{(l)})$ and ${\rm Aut}(x_i)\alpha_{ij}^{(l)} {\rm Aut}(x_j)$ are isomorphic. Indeed by (EC3), the following biset map
$$\phi\colon X(\alpha_{ij}^{(l)})\longrightarrow {\rm Aut}(x_i)\alpha_{ij}^{(l)} {\rm Aut}(x_j), \quad (\eta_i^a, \eta_j^b)\mapsto \eta_i^a\circ \alpha_{ij}^{(l)}\circ \eta_j^b$$
is well-defined, which is surjective. The cardinality of $X(\alpha_{ij}^{(l)})$ equals $\frac{c_ic_j}{{\rm gcd}(c_i, c_j)}$.   By (EC2), both $c_i$ and $c_j$ divide the cardinality of ${\rm Aut}(x_i)\alpha_{ij}^{(l)} {\rm Aut}(x_j)$. It follows that $\phi$ is bijective. This proves the claim,  and completes the proof.
\end{proof}

\subsection{The construction of $(C', D', \Omega')$}\label{sub:4.2}

In what follows, we assume that the base field $k$ has \emph{enough roots of unity}, that is, any polynomial $t^a-1$ in $k[t]$ splits for each $a\geq 2$.

We assume first that  ${\rm char} k=p>0$. We fix a Cartan triple $(C, D, \Omega)$ as above. Assume that $$c_i=p^{r_i}d_i$$
 with $r_i\geq 0$ and ${\rm gcd}(p, d_i)=1$.

Set $m=\sum_{i=1}^n d_i$, which is the cardinality of the set
$$M=\bigsqcup_{1\leq i\leq n} \{(i, l_i)\; |\; 0\leq l_i<d_i\}.$$
For each $1\leq i, j\leq n$, we set
\[\Sigma_{ij}=\{ (l_i,l_j)\; |\;  0\leq l_i< d_i, 0\leq l_j< d_j, l_ip^{r_i}\equiv l_jp^{r_j} ({\rm mod}\ {\rm gcd}(d_i, d_j))\}.\]
We will define a Cartan matrix $C'\in M_m(\mathbb{Z})$, whose rows and columns are indexed by $M$. The diagonal entries of $C'$ are $2$, while the off-diagonal entries are given as follows
 \[c'_{(i, l_i), (j, l_j)}=\begin{cases}
-{\rm gcd}(c_{ij},c_{ji}) p^{r_j-{\rm min}(r_i, r_j)}, & \text{ if } (l_i, l_j)\in \Sigma_{ij}; \\
0, & \text{otherwise.}
\end{cases}\]
In particular, $c'_{(i, l_i), (i, l'_i)}=0$ for $l_i\neq l'_i$. Let $D'$ be a diagonal matrix of rank $m$, whose $(i, l_i)$-th component is given by $p^{r_i}$. Since
$$p^{r_i}c'_{(i, l_i),(j, l_j)}=p^{r_j} c'_{(j, l_j),(i, l_i)},$$
the product matrix $D'C'$ is symmetric. Set
$$\Omega'=\{((i, l_i),(j, l_j))\; |\; (i, j)\in \Omega, (l_i, l_j)\in \Sigma_{ij}\},$$
which is an orientation of $C'$. Then we have the required Cartan triple $(C', D', \Omega')$. We point out that a similar construction is given in \cite[Section 14.1]{Lus} and \cite[Lemma~21]{Hu04}. We refer to Section \ref{sec:6} for further study on this construction.

We observe that the above definitions work well if ${\rm char}k=0$, in which case we put $r_i=0$ and $d_i=c_i$. Then the Cartan matrix $C'$ is symmetric and the symmetrizer $D'$  is the identity matrix.

The main result provides an explicit isomorphism between certain category algebras and the algebras studied in \cite{GLS}. It motivates the notion of a free EI category of Cartan type.

 \begin{thm}\label{thm:main}
  Let $(C, D, \Omega)$ be a Cartan triple and   $\mathcal{C}(C, D, \Omega)$ as in Definition $\ref{defn:EI}$. Assume that $k$ has enough roots of unity. Consider the Cartan triple $(C', D', \Omega')$ as above and the algebra $H(C', D', \Omega')$ in Definition $\ref{defn:GLS}$. Then there is an isomorphism of algebras
  $$k\mathcal{C}(C, D, \Omega)\simeq H(C', D', \Omega').$$
 \end{thm}

\begin{proof}
We assume that ${\rm char}k=p>0$. Set $\mathcal{C}=\mathcal{C}(C, D, \Omega)$ and $H=H(C', D', \Omega')$.

We observe that for $((i, l_i), (j, l_j))\in \Omega'$, we have
\begin{align}\label{equ:gcd}
{\rm gcd}(c'_{(i, l_i), (j, l_j)}, c'_{(j, l_j), (i, l_i)})={\rm gcd}(c_{ij}, c_{ji}).
\end{align}
By Proposition \ref{prop:iso}, there is an isomorphism of algebras
$$ k\mathcal{C} \stackrel{\sim}\longrightarrow T_A(V),$$
where $A=\prod_{i=1}^{n} kX(i)$ and
\begin{align*}
V&=\bigoplus_{(i,j)\in \Omega} \bigoplus_{g=1}^{{\rm gcd}(c_{ij}, c_{ji})} k(X(i)\times_{G_{ij}} X(j))\\
& \simeq \bigoplus_{(i,j)\in \Omega} \bigoplus_{g=1}^{{\rm gcd}(c_{ij}, c_{ji})} kX(i)\otimes_{kG_{ij}} kX(j).
\end{align*}
Here, the last isomorphism follows from Lemma \ref{lem:bi-prod}.

By Proposition \ref{prop:H}, there is another isomorphism of algebras
$$H\stackrel{\sim}\longrightarrow T_B(W),$$
where $B=\prod_{i=1}^n \prod_{l_i=0}^{d_i-1} H_{(i, l_i)}$ and
\begin{align*}
W &=\bigoplus_{((i, l_i), (j, l_j))\in \Omega'} \bigoplus_{g=1}^{{\rm gcd}(c'_{(i, l_i), (j, l_j)}, c'_{(j, l_j), (i, l_i)})} H_{(i, l_i)}\otimes_{H'_{ij}} H_{(j, l_j)}\\
&= \bigoplus_{(i,j)\in \Omega} \bigoplus_{(l_i, l_j)\in \Sigma_{ij}} \bigoplus_{g=1}^{{\rm gcd}(c_{ij}, c_{ji})} H_{(i, l_i)}\otimes_{H'_{ij}} H_{(j, l_j)}\\
& = \bigoplus_{(i,j)\in \Omega} \bigoplus_{g=1}^{{\rm gcd}(c_{ij}, c_{ji})}    \bigoplus_{(l_i, l_j)\in \Sigma_{ij}} H_{(i, l_i)}\otimes_{H'_{ij}} H_{(j, l_j)}.
\end{align*}
The second equality uses (\ref{equ:gcd}). Here, we recall that $H_{(i, l_i)}=k[\epsilon_{(i, l_i)}]/(\epsilon_{(i, l_i)}^{p^{r_i}})$, and that $H'_{ij}=k[\epsilon_{ij}]/(\epsilon_{ij}^{p^{{\rm min}(r_i, r_j)}})$, which coincides with $H_{(i, l_i),(j, l_j)}$ in the notation of Section 3.

To obtain the required isomorphism, it suffices to have an isomorphism $A\simeq B$ of algebras and an isomorphism $V\simeq W$ of $A$-$A$-bimodules, where the $A$-$A$-bimodule structure on $W$ is induced by the given algebra isomorphism.

For the two isomorphisms, we just claim that there is an isomorphism of algebras
$$\Theta_i\colon kX(i)\stackrel{\sim}\longrightarrow \prod_{l_i=0}^{d_i-1} H_{(i, l_i)} $$
and an isomorphism of $kX(i)$-$kX(j)$-bimodules
\begin{align}\label{iso:bimod}
kX(i)\otimes_{kG_{ij}} kX(j)\stackrel{\sim}\longrightarrow   \bigoplus_{(l_i, l_j)\in \Sigma_{ij}} H_{(i, l_i)}\otimes_{H'_{ij}} H_{(j, l_j)},
\end{align}
for any $1\leq i, j\leq n$. We emphasize that the $kX(i)$-$kX(j)$-bimodule structure on $H_{(i, l_i)}\otimes_{H'_{ij}} H_{(j, l_j)}$ is induced by the isomorphisms $\Theta_i$ and $\Theta_j$.

Since $X(i)=\langle \eta_i\; |\; \eta_i^{c_i}=1\rangle$, we identify the group algebra $kX(i)$  with $k[\eta_i]/(\eta_i^{c_i}-1)$. Then the claim follows immediately from Proposition \ref{prop:main} below.

The proof in  the characteristic zero case is very similar. We omit the details.
\end{proof}

\section{An explicit decomposition of a bimodule}

In this section, we give an explicit decomposition of a tensor bimodule needed at the end of the proof of Theorem \ref{thm:main}; see Proposition \ref{prop:main}. The decomposition seems to be elementary, but somehow technical.

Let $k$ be a field having enough roots of unity with ${\rm char}k=p>0$. Let $a, b$ be two positive integers. Write $a=p^ra'$ and $b=p^sb'$ with $r, s\geq 0$ and ${\rm gcd}(p, a')=1={\rm gcd}(p, b')$. Then we have ${\rm gcd}(a, b)=p^{{\rm min}(r, s)} {\rm gcd}(a', b')$.

 Let $A=k[x]/(x^a-1)$, $B=k[y]/(y^b-1)$ and $C=k[z]/(z^{{\rm gcd}(a, b)}-1)$. There are injective algebra homomorphisms $\iota_1\colon C\rightarrow A$ and $\iota_2\colon C\rightarrow B$ such that $\iota_1(z)=x^{\frac{a}{{\rm gcd}(a, b)}}$ and $\iota_2(z)=y^{\frac{b}{{\rm gcd}(a, b)}}$. We are mainly concerned with the tensor $A$-$B$-bimodule $A\otimes_CB$.

We fix an $a'$-th primitive root $\zeta_1$ of unity and a $b'$-th primitive root $\zeta_2$ of unity such that $\zeta_1^{\frac{a'}{{\rm gcd}(a', b')}}=\zeta_2^{\frac{b'}{{\rm gcd}(a', b')}}$; the common value will be denoted by $\zeta$, which is a ${\rm gcd}(a', b')$-th primitive root of unity.

We have the following well-known isomorphisms of algebras
\[\pi_A \colon A\stackrel{\sim}{\longrightarrow} \prod_{i=0}^{a'-1} k[x]/{(x-\zeta_1^i)^{p^r}}=\prod_{i=0}^{a'-1} A_i,\]

\[\pi_B \colon B\stackrel{\sim}{\longrightarrow} \prod_{j=0}^{b'-1} k[y]/{(y-\zeta_2^j)^{p^s}}=\prod_{j=0}^{b'-1} B_j,\]
and
\[\pi_C \colon C\stackrel{\sim}{\longrightarrow} \prod_{l=0}^{{\rm gcd}(a', b')-1} k[z]/{(z-\zeta^l)^{p^{{\rm min}(r, s)}}}=\prod_{l=0}^{{\rm gcd}(a', b')-1} C_l,\]
each of which sends $\bar{f}$ to $(\bar{f}, \cdots, \bar{f})$ for any polynomial $f$.

For each $0\leq l< {\rm gcd}(a', b')$, we set
$$\Sigma(l)=\{m\;|\; 0\leq m< a', mp^r\equiv lp^{{\rm min}(r, s)} ({\rm mod}\;  {\rm gcd}(a', b')) \}.$$
Then we have a disjoint union
\begin{align}\label{equ:dis1}
\{0, 1, \cdots, a'-1\}=\bigsqcup_{l=0}^{{\rm gcd}(a', b')-1} \Sigma(l).
\end{align}
Here, we use the fact that $\bar{p}$ is invertible in $\mathbb{Z}_{{\rm gcd}(a', b')}$.
Similarly, we set
$$\Sigma'(l)=\{m\;|\; 0\leq m< b', mp^s\equiv lp^{{\rm min}(r, s)} ({\rm mod}\;  {\rm gcd}(a', b')) \},$$
and have the disjoint union
$$\{0, 1, \cdots, b'-1\}=\bigsqcup_{l=0}^{{\rm gcd}(a', b')-1} \Sigma'(l).$$

\begin{lem}
Assume that $m\in \Sigma(l)$. Then the following identity
$$(x^\frac{a}{{\rm gcd}(a, b)}-\zeta^l)^{p^{{\rm min}(r, s)}}=0$$
holds in the algebra $A_m$.
\end{lem}

\begin{proof}
We observe that $\frac{a}{{\rm gcd}(a, b)}p^{{\rm min}(r, s)}=\frac{a'}{{\rm gcd}(a', b')}p^r$ and $\zeta^{lp^{{\rm min}(r, s)}}=\zeta^{mp^r}$ by the fact that $m\in \Sigma(l)$. Then we have the first equality in the following identity
\begin{align*}
(x^\frac{a}{{\rm gcd}(a, b)}-\zeta^l)^{p^{{\rm min}(r, s)}}&=x^{\frac{a'}{{\rm gcd}(a', b')}p^r}-\zeta^{mp^r}\\
                                                           &=(x^{\frac{a'}{{\rm gcd}(a', b')}}-\zeta^m)^{p^r}\\
                                                           &=(x^{\frac{a'}{{\rm gcd}(a', b')}}-(\zeta_1^m)^{\frac{a'}{{\rm gcd}(a', b')}})^{p^r}.
\end{align*}
We observe that $(x-\zeta_1^m)^{p^r}$ is a factor of the above identity. Therefore, it is zero in $A_m$.
\end{proof}

The above lemma implies that the following algebra homomorphism
$$\phi_l\colon C_l\longrightarrow \prod_{i\in \Sigma(l)} A_i, \quad z\mapsto (x^\frac{a}{{\rm gcd}(a, b)}, \cdots, x^\frac{a}{{\rm gcd}(a, b)}).$$
is well defined for each $0\leq l< {\rm gcd}(a', b')$.  By the disjoint union (\ref{equ:dis1}), we have the following algebra homomorphism
$$(\phi_l)\colon \prod_{l=0}^{{\rm gcd}(a', b')-1} C_l\longrightarrow \prod_{l=0}^{{\rm gcd}(a', b')-1} (\prod_{i\in \Sigma(l)}A_i)=\prod_{i=0}^{a'-1} A_i.$$

By a similar argument, we have a well-defined algebra homomorphism
$$\psi_l\colon C_l\longrightarrow \prod_{j\in \Sigma'(l)} B_j, \quad z\mapsto (y^\frac{b}{{\rm gcd}(a, b)}, \cdots, y^\frac{b}{{\rm gcd}(a, b)}).$$
for each $0\leq l< {\rm gcd}(a', b')$. Summing them up, we have an algebra homomorphism
$$(\psi_l)\colon \prod_{l=0}^{{\rm gcd}(a', b')-1} C_l\longrightarrow \prod_{l=0}^{{\rm gcd}(a', b')-1} (\prod_{j\in \Sigma'(l)}B_j)=\prod_{j=0}^{b'-1} B_j.$$

\begin{lem}\label{lem:comm-diag1} Keep the notation as above. Then the following diagram commutes
\[\xymatrix@C=1.5cm{
	& A \ar[d]_{\pi_A} & C \ar[l]_{\iota_1} \ar[r]^{\iota_2} \ar[d]^{\pi_C} & B \ar[d]^{\pi_B}  \\
	&\prod_{i=0}^{a'-1} A_i  & \prod_{l=0}^{{\rm gcd}(a', b')-1} C_l \ar[l]_{(\phi_l)} \ar[r]^{(\psi_l)}   & \prod_{j=0}^{b'-1} B_j.}\]
\end{lem}

\begin{proof}
It suffices to observe that the commutativity holds clearly for the generator $z$ of $C$.
\end{proof}

For each $0\leq i< a'$ and $0\leq j <b'$, we set $A'_i:=k[\epsilon_i]/(\epsilon_i^{p^r})$, $B'_j=k[\eta_j]/(\eta_j^{p^s})$ and $C'=k[\omega]/(\omega^{p^{{\rm min}(r, s)}})$.  There are injective algebra homomorphisms  $\phi'_{i}\colon C'\rightarrow A'_i$ and $\psi'_j\colon C'\rightarrow B'_j$ given by $\phi'_i(\omega)=\epsilon_i^{p^{r-{\rm min}(r, s)}}$ and $\psi'_j(\omega)=\eta_j^{p^{s-{\rm min}(r, s)}}$. Consequently, we have the tensor $A'_i$-$B'_j$-bimodule $A'_i\otimes_{C'}B'_j$.

We observe an isomorphism
$$\theta_i\colon A'_i\longrightarrow A_i$$
 of algebras sending $\epsilon_i$ to $x^{a'}-1$. We remark that there is another obvious isomorphism $A'_i\rightarrow A_i$ sending $\epsilon_i$ to $x-\zeta_1^{i}$. However, we really have to choose the isomorphism $\theta_i$ making Lemma \ref{lem:comm-diag2} work.

Similarly, we have isomorphisms $\rho_j\colon B'_j\rightarrow B_j$ and $\gamma_l\colon C'\rightarrow C_l$ given by $\rho_j(\eta_j)=y^{b'}-1$ and $\gamma_l(\omega)=z^{{\rm gcd}(a', b')}-1$.

\begin{lem}\label{lem:comm-diag2}
Keep the notation as above. We assume that $i\in \Sigma(l)$ and $j\in \Sigma'(l)$ for some $l$. Then the following diagram commutes
\[
\xymatrix@C=2.0cm{
& A'_i \ar[d]_{\theta_i} & C' \ar[l]_{\phi'_i} \ar[r]^{\psi'_j} \ar[d]^{\gamma_l} & B'_j \ar[d]^{\rho_j}  \\
& A_i  & C_l \ar[l]_{p_i\circ \phi_l} \ar[r]^{p'_j\circ \psi_l}   & B_j,
}
\]
where $p_i : \prod_{m\in \Sigma(l)} A_m\twoheadrightarrow A_i$ and $p'_j : \prod_{m\in \Sigma'(l)} B_m\twoheadrightarrow B_j$ are the canonical projections.
\end{lem}

\begin{proof}
We just observe that the commutativity holds clearly for the generator $\omega$ of $C'$. Here, we use the specific choice of $\theta_i$, $\gamma_l$ and $\rho_j$.
\end{proof}

The main result in this section claims a decomposition of the tensor $A$-$B$-bimodule $A\otimes_C B$. We set
 $$\Sigma=\{(i, j)\; |\; 0\leq i<a', 0\leq j<b', ip^{r}\equiv jp^s ({\rm mod}\; {\rm gcd}(a', b'))\}.$$

\begin{prop}\label{prop:main}
Keep the notation as above. Then the following statements hold.
\begin{enumerate}
\item There are isomorphisms of algebras
$$(\prod_{i=0}^{a'-1}\theta_i)^{-1}\circ \pi_A\colon A\longrightarrow \prod_{i=0}^{a'-1}A'_i \quad {\rm and} \quad (\prod_{j=0}^{b'-1}\rho_j)^{-1}\circ \pi_B\colon B\longrightarrow \prod_{j=0}^{b'-1} B'_j.$$
\item There is an isomorphism of $A$-$B$-bimodules
$$A\otimes_C B\stackrel{\sim}\longrightarrow \bigoplus_{(i, j)\in \Sigma} A'_i\otimes_{C'} B'_j$$
where the $A$-$B$-bimodule structure on $A'_i\otimes_{C'} B'_j$ is induced by the algebra isomorphisms in (1).
\end{enumerate}
\end{prop}

\begin{proof}
It suffices to prove (2). The vertical morphisms in the commutative diagrams of the above two lemmas are all isomorphisms. By Lemma \ref{lem:comm-diag1}, we have the first isomorphism in the following identity
\begin{align*}
A\otimes_C B & \simeq \bigoplus_{l=0}^{{\rm gcd}(a', b')-1} (\prod_{i\in \Sigma(l)} A_i)\otimes_{C_l}(\prod_{j\in \Sigma'(l)} B_j)\\
&\simeq \bigoplus_{(i, j)\in \Sigma} A_i\otimes_{C_l} B_j\\
&\simeq \bigoplus_{(i, j)\in \Sigma} A'_i\otimes_{C'} B'_j.
\end{align*}
The last isomorphism follows from Lemma \ref{lem:comm-diag2}.
\end{proof}

Choose $A, B$ and $C$ to be $kX(i)$, $kX(j)$ and $kG_{ij}$, respectively. Then we get the isomorphism (\ref{iso:bimod}) in the proof of Theorem \ref{thm:main}.

\section{The construction in Subsection \ref{sub:4.2} and examples}\label{sec:6}

In this section, we  study the construction from $(C, D, \Omega)$ to  $(C', D', \Omega)$ in Subsection \ref{sub:4.2}. We prove that the valued graph $\Gamma$ of $C$ is a disjoint union of Dynkin graphs (\emph{resp.} Euclidean graphs, graphs of indefinite type) if and only if so is the valued graph $\Gamma'$ of $C'$; see Proposition \ref{prop:type}.

We observe that our construction is related to the well-known correspondence in \cite[Section 14.1]{Lus} between Cartan matrices and graphs with automorphisms. Consequently, in the hereditary case, the isomorphism in Theorem \ref{thm:main} yields an algebraic enrichment of the mentioned correspondence; see Proposition \ref{prop:compa}.

\subsection{The types of Cartan matrices}

Let $C=(c_{ij})\in M_n(\mathbb{Z})$ be a Cartan matrix. Denote by $\Gamma$ the associated \emph{valued graph}. Recall that the vertices of $\Gamma$ is given by $\{1, 2, \cdots, n\}$ and that there is a unique edge between $i$ and $j$ if and only if $c_{ij}<0$; moreover, we put a valuation $(|c_{ij}|, |c_{ji}|)$ on the edge. We emphasize that the valuation is not assigned to the edge, but really to the ordered pair $(i, j)$ of vertices. We say that $C$ is \emph{connected} provided that $\Gamma$ is a connected graph. In what follows, we will identify the Cartan matrix $C$ with its valued graph $\Gamma$. In particular, symmetric Cartan matrices correspond to usual graphs, which do not have loops, but possibly have parallel edges.

Let $D={\rm diag}(c_1, c_2, \cdots, c_n)$ be a symmetrizer of $C$.  Denote by $\mathbb{Z}^n$ the root lattice of $C$, which has a canonical basis $\{e_1, e_2, \cdots, e_n\}$. Recall that the quadratic form $q_C\colon \mathbb{Z}^n\rightarrow \mathbb{Z}$ is given by $q_C({\bf x})=\sum_{i=1}^n c_ix_i^2+\sum_{i<j} c_i c_{ij}x_ix_j$ for ${\bf x}=(x_1, x_2, \cdots, x_n)$.

We fix a prime number $p>0$. Recall from Subsection \ref{sub:4.2} that $(C', D', \Omega')$ is constructed from $(C, D, \Omega)$. We denote by $\Gamma'$ its valued graph. The vertices of $\Gamma'$ is indexed by the following set
$$M=\bigsqcup_{1\leq i\leq n} \{(i, l_i)\; |\; 0\leq l_i<d_i\}.$$
Here, we recall that $c_i=p^{r_i}d_i$. Its root lattice is denoted by $\mathbb{Z}^M$ with a canonical basis $\{e_{(i, l_i)} \mid (i, l_i)\in M \}$ and its quadratic form $q_{C'}$.

The following result is inspired by \cite[Subsection 13.2.9]{Lus}.

\begin{lem}\label{lem:quad}
The natural injection $\theta\colon \mathbb{Z}^n\rightarrow \mathbb{Z}^M$ sending each $e_i$ to $\sum_{0\leq l_i<d_i} e_{(i, l_i)}$ is compatible with the quadratic form, that is, $q_C({\bf x})=q_{C'}(\theta({\bf x}))$.
\end{lem}

\begin{proof}
Denote by $(-, -)_C$ and $(-, -)_{C'}$ the symmetric bilinear form on $\mathbb{Z}^n$ and $\mathbb{Z}^M$ corresponding to $q_C$ and $q_{C'}$, respectively. It suffices to claim that $(e_i, e_j)_C=(\theta(e_i), \theta(e_j))_{C'}$.

Recall that $(e_i, e_i)_C=2c_i=2p^{r_i}d_i$ and $(e_i, e_j)_C=c_ic_{ij}$ for $i\neq j$. Similarly, $(e_{(i, l_i)}, e_{(i, l_i)})_{C'}=2p^{r_i}$ and $(e_{(i, l_i)}, e_{(j, l_j)})_{C'}=p^{r_i}c'_{(i, l_i), (j, l_j)}$. In particular, $(e_{(i, l_i)}, e_{(i, l'_i)})_{C'}=0$ for $l_i\neq l'_i$. Then the required identity is easily verified for the case $i=j$.

We assume that $i\neq j$. Then we have
\begin{align*}
(\theta(e_i), \theta(e_j))_{C'}&=\sum_{0\leq l_i< d_i, 0\leq l_j< d_j} p^{r_i}c'_{(i, l_i), (j, l_j)}\\
                               &=-\sum_{(l_i, l_i)\in \Sigma_{ij}} p^{r_i+r_j-{\rm min}(r_i, r_j)} {\rm gcd}(c_{ij}, c_{ji})\\
                               &=-  p^{r_i+r_j-{\rm min}(r_i, r_j)} \frac{d_id_j}{{\rm gcd}(d_i, d_j)} {\rm gcd}(c_{ij}, c_{ji}).
\end{align*}
where the last equality uses the fact that the cardinality of $\Sigma_{ij}$ is $\frac{d_id_j}{{\rm gcd}(d_i, d_j)}$. We assume without loss of generality that $r_i\geq r_j$. Since $c_ic_{ij}=c_jc_{ji}$, we have
$$p^{r_i-r_j} \frac{d_i}{{\rm gcd}(d_i, d_j)} \frac{c_{ij}}{{\rm gcd}(c_{ij}, c_{ji})}=\frac{d_j}{{\rm gcd}(d_i, d_j)} \frac{c_{ji}}{{\rm gcd}(c_{ij}, c_{ji})}.$$
Since $d_i$ and $d_j$ are coprime to $p$, we infer that
$$ \frac{-c_{ij}}{{\rm gcd}(c_{ij}, c_{ji})}=\frac{d_j}{{\rm gcd}(d_i, d_j)}.$$
Then the required equality in the claim follows immediately.
\end{proof}

Recall that the symmetrizer $D$ is \emph{minimal},  provided that ${\rm gcd}(c_1, c_2, \cdots, c_n)=1$. For a connected Cartan matrix, the minimal symmetrizer is unique and any symmetrizer is a multiple of the minimal one.

\begin{lem}\label{lem:conn}
Let $C=(c_{ij})\in M_n(\mathbb{Z})$ be  a connected Cartan matrix with its minimal symmetrizer $D$. Then the constructed Cartan matrix $C'$ is also connected.
\end{lem}

\begin{proof}
For two distinct vertices $(i, l)$ and $(j, l')$ in $\Gamma'$, we will construct a path connecting them. By the connectedness of $\Gamma$, there is a path in $\Gamma$
$$i=x(1) \text{---} x(2) \text{---} \cdots \text{---} x(m)\text{---} x(m+1)=j$$
such that for each $1\leq s\leq n$, there exists at least one $a$ satisfying $x(a)=s$.  For each $1\leq a\leq m$, we set $g_a={\rm gcd}(d_{x(a)}, d_{x(a+1)})$. By the minimality of $D$, we have ${\rm gcd}(g_1, g_2, \cdots, g_{m})=1$. There exist $s_a\in \mathbb{Z}$ such that $\sum_{a=1}^m s_ag_a=1$.

Set $l_1=l$. For $2\leq a\leq m+1$, we define $l_a$  to be the  unique integer satisfying $0\leq l_a<d_{x(a)}$ and
$$l_{a}p^{r_{x(a)}}\equiv lp^{r_i}+\sum_{b=1}^{a-1} s_bg_b(l'p^{r_j}-lp^{r_i}) ({\rm mod}\; d_{x(a)}).$$
We observe $l_{m+1}=l'$. Moreover, it is clear that $l_{a}p^{r_{x(a)}}\equiv l_{a+1}p^{r_{x(a+1)}} ({\rm mod}\;  g_a)$ for each $1\leq a\leq m$. Therefore, there is an edge from $(x(a), l_a)$ to $(x(a+1), l_{a+1})$ in $\Gamma'$. This yields the required construction.
\end{proof}

 The following result implies that in the construction of Subsection \ref{sub:4.2},  we may assume that $D$ is minimal.

\begin{lem}\label{lem:mul}
Let $c=p^rd$ be a positive integer with $r\geq 0$ and ${\rm gcd}(p,d)=1$. Denote by  $(C'', D'', \Omega'')$ the Cartan triple constructed from $(C, cD, \Omega)$ as in Subsection $\ref{sub:4.2}$, and by $\Gamma''$  its valued graph. Then $\Gamma''$ is isomorphic to the disjoint union of $d$ copies of $\Gamma'$.
\end{lem}

\begin{proof}
The  vertices of $\Gamma''$ are indexed by $M'=\bigsqcup_{1\leq i\leq n} \{(i, m_i)\; |\; 0\leq m_i< dd_i\}.$
The following set
$$\{(m_i, m_j)\; |\; 0\leq m_i< dd_i, 0\leq m_j< dd_j, m_ip^{r_i+r}\equiv m_jp^{r_j+r} ({\rm mod}\;  d{\rm gcd}(d_i, d_j))\}$$
will be denoted by $\Sigma'_{ij}$. Then the off-diagonal entries of $C''$ are given as follows:
 \[c''_{(i, m_i), (j, m_j)}=\begin{cases}
-{\rm gcd}(c_{ij},c_{ji}) p^{(r_j+r)-{\rm min}(r_i+r, r_j+r)}, & \text{ if } (m_i, m_j)\in \Sigma'_{ij}; \\
0, & \text{otherwise.}
\end{cases}\]
It follows that a off-diagonal entry $c''_{(i, m_i), (j, m_j)}\neq 0$ implies that $i\neq j$ and $m_ip^{r_i}\equiv m_j p^{r_j} ({\rm mod}\ d)$, in which case it is equal to $c'_{(i, l_i), (j, l_j)}$ for $(l_i, l_j)\in \Sigma_{ij}$.

For each $0\leq a <d$, we denote by $\Gamma''_a$ the full subgraph of $\Gamma''$ formed by these vertices $\{(i, m_i)\in M'\;| \; m_ip^{r_i}\equiv a ({\rm mod}\ d)\}$. It follows that $\Gamma''$ is a disjoint union of these subgraphs $\Gamma''_a$'s. We observe that each $\Gamma''_a$ is isomorphic to $\Gamma'$ by sending $(i, m_i)$ to $(i, l_i)$, where $l_i$ is uniquely determined by the constraint
 $$\frac{m_ip^{r_i}-a}{d}\equiv l_ip^{r_i} \; ({\rm mod}\; d_i).$$
 Here for the isomorphism, we use the fact that $(m_i, m_j)$ lies in $\Sigma'_{ij}$ if and only if  $(l_i, l_j)$ lies in $\Sigma_{ij}$.
\end{proof}

The following fact follows immediately from the above two lemmas and the fact that  any symmetrizer of a connected Cartan matrix is a multiple of the minimal symmetrizer.

\begin{cor}\label{cor:mul}
Keep the notation as above. Assume that $C$ is  connected. Then the connected components of the valued graph $\Gamma'$ are isomorphic to each other. \hfill $\square$
\end{cor}

Recall that a connected Cartan matrix $C$ is a \emph{Dynkin graph} (\emph{resp.} an \emph{Euclidean graph},  a \emph{graph of indefinite type}), provided that its quadratic form $q_C$ is positive definite (\emph{resp.} positive semi-definite, indefinite).

\begin{prop}\label{prop:type}
Keep the notation as above. Assume further that $C$ and thus $\Gamma$ are connected. Then $\Gamma$ is a Dynkin graph (\emph{resp.} an Euclidean graph,  a graph of indefinite type) if and only if $\Gamma'$ is a disjoint union of Dynkin graphs (\emph{resp.}  Euclidean graphs,   graphs of indefinite type), whose connected components are of the same type.
\end{prop}

\begin{proof}
We first prove the statement on Dynkin graphs. The  ``if" part follows from Lemma \ref{lem:quad}. For the ``only if" part, we recall that the symmetrizer $D$ is a multiple of the minimal one. Then by Lemma \ref{lem:mul} we may assume that the symmetrizer $D$ is minimal. If $C$ is symmetric and then $D$ is the identity matrix, then $C'=C$. For non-symmetric Cartan matrices of Dynkin type and minimal symmetrizers, the description of $\Gamma'$ is given in the third column of Table 1 below. We observe that $\Gamma'$ is necessarily a Dynkin graph.

Using the proved statement on Dynkin graphs and Table 2, we infer the statement on Euclidean graphs with the same argument. Then the remaining statement on indefinite types follows immediately.
\end{proof}

\begin{rem}
(1) We observe that usually $\Gamma$ and the connected components of $\Gamma'$  are of  different types; see Tables 1 and 2 below.

(2) Proposition \ref{prop:type} also hold for  $p=0$ along the same line of reasoning. In the above notation, we have $d_i=c_i$,  $r_i=0=r$ and $d=c$.
\end{rem}

We mention that  Tables 1 and 2 are essentially contained in \cite[Section 14.1]{Lus} with a completely different terminology. For the lists of Dynkin graphs and Euclidean graphs, we refer to \cite[VII.3]{ARS}.

The following example unifies the actual computation in establishing Tables 1 and 2.

\begin{exm}
{\rm Let $C$ be a connected Cartan matrix such that its minimal symmetrizer $D$ equals ${\rm diag}(q^{a_1}, \cdots, q^{a_n})$, where $q$ is  a prime number  and $a_i\geq 0$. We observe that
		$$c_{ij}=-{\rm gcd}(c_{ij}, c_{ji}) q^{a_j-{\rm min}(a_i, a_j)}, i\neq j.$$
Consider the construction of $(C', D', \Omega')$, where the characteristic  of the field is $p$. It follows that  $(C, D, \Omega)=(C', D', \Omega')$ and $\Gamma'=\Gamma$ provided that $p=q$ .

		We assume that $p\neq q$. Then $C'$ is symmetric and the symmetrizer $D'$ is the identity matrix. In other words, the valued graph $\Gamma'$ of $C'$ is a usual graph. Moreover, as we will see in Subsection 6.2 below, the graph $\Gamma'$ has an admissible automorphism $\sigma$, listed in the rightmost column in the tables. We observe that, the graphs $(\Gamma',\sigma)$ with automorphisms,  listed in Tables 1 and 2,  are due to \cite[14.1.5 (a)-(d) and  14.1.6 (a)-(i)]{Lus}. }
\end{exm}

\begin{table}[ht]
\renewcommand\arraystretch{1.5}
\caption{Non-simply-laced Dynkin graphs $\Gamma$}
\begin{tabular}{|c|c|c|c|}
\hline
$\Gamma$ & \thead{minimal \\ symmetrizer $D$} & $\Gamma'$ &  \thead{$(\Gamma', \sigma)$ in \cite{Lus}} \\
\hline
\multirow{2}{*}{\thead{\begin{tikzpicture}
		\draw (-0.2,0) node[anchor=east]  {$B_n$ : };
		\draw (0,0) -- (0.6,0) (1.8,0)-- (2.4,0);
		\draw[fill] (1,0) circle [radius=0.025];
		\draw[fill] (1.2,0) circle [radius=0.025];
		\draw[fill] (1.4,0) circle [radius=0.025];
		\draw[fill] (0,0) circle [radius=0.05];
		\draw[fill] (0.6,0) circle [radius=0.05];
		\draw[fill] (1.8,0) circle [radius=0.05];
		\draw[fill] (2.4,0) circle [radius=0.05];
		\node [below] at (3.3,0.2) {($n\geq 3$)};
		\node [below] at (0.3,0.4) {\tiny{(1,2)}};
		\node [below] at (0,0) {1};
		\node [below] at (0.6,0) {2};
		\node [below] at (2.4,0) {n};
		\end{tikzpicture}}} &  	\multirow{2}{*}{\thead{${\rm diag}(2,1,\cdots,1)$}} &  If $p=2$, $B_n$ & \\ \cline{3-4} & &  If $p\neq 2$, $D_{n+1}$  & \cite[14.1.6 (b)]{Lus}\\
\hline
\multirow{2}{*}{\thead{\begin{tikzpicture}
		\draw (-0.2,0) node[anchor=east]  {$C_n$ : };
		\draw (0,0) -- (0.6,0) (1.8,0)-- (2.4,0);
		\draw[fill] (1,0) circle [radius=0.025];
		\draw[fill] (1.2,0) circle [radius=0.025];
		\draw[fill] (1.4,0) circle [radius=0.025];
		\draw[fill] (0,0) circle [radius=0.05];
		\draw[fill] (0.6,0) circle [radius=0.05];
		\draw[fill] (1.8,0) circle [radius=0.05];
		\draw[fill] (2.4,0) circle [radius=0.05];
		\node [below] at (3.3,0.2) {($n\geq 2$)};
		\node [below] at (0.3,0.4) {\tiny{(2,1)}};
		\node [below] at (0,0) {1};
		\node [below] at (0.6,0) {2};
		\node [below] at (2.4,0) {n};
		\end{tikzpicture}}}&  \multirow{2}{*}{\thead{${\rm diag}(1,2,\cdots,2)$}}&
If $p=2$, $C_n$ & \\ \cline{3-4} & &  If $p\neq 2$, $A_{2n-1}$  & \cite[14.1.6 (a)]{Lus}\\
\hline
\multirow{2}{*}{\thead{\begin{tikzpicture}
		\draw (-0.2,0) node[anchor=east]  {$F_4$ : };
		\draw (0,0) -- (0.6,0)-- (1.2,0)-- (1.8,0);
		\draw[fill] (0,0) circle [radius=0.05];
		\draw[fill] (0.6,0) circle [radius=0.05];
		\draw[fill] (1.2,0) circle [radius=0.05];
		\draw[fill] (1.8,0) circle [radius=0.05];
		\node [below] at (0.9,0.4) {\tiny{(1,2)}};
		\node [below] at (0,0) {1};
		\node [below] at (0.6,0) {2};
		\node [below] at (1.2,0) {3};
		\node [below] at (1.8,0) {4};
		\end{tikzpicture}}} &
\multirow{2}{*}{\thead{${\rm diag}(2,2,1,1)$ }}&
If $p=2$, $F_4$ & \\ \cline{3-4} & &  If $p\neq 2$, $E_6$  & \cite[14.1.6 (d)]{Lus}\\
\hline
\multirow{2}{*}{\thead{	\begin{tikzpicture}
		\draw (-0.2,0) node[anchor=east]  {$G_2$ : };
		\draw (0,0) -- (0.6,0);
		\draw[fill] (0,0) circle [radius=0.05];
		\draw[fill] (0.6,0) circle [radius=0.05];
		\node [below] at (0.3,0.4) {\tiny{(1,3)}};
		\node [below] at (0,0) {1};
		\node [below] at (0.6,0) {2};
		\end{tikzpicture}}} & \multirow{2}{*}{\thead{${\rm diag}(3,1)$}}  &  If $p=3$, $G_2$  & \\ \cline{3-4} & &  If $p\neq 3$, $D_4$  & \cite[14.1.6 (c)]{Lus}\\
\hline
\end{tabular}
\end{table}

\begin{table}[ht]

\renewcommand\arraystretch{1.95}
\caption{Non-simply-laced Euclidean graphs $\Gamma$ (with $\tilde{A}_{12}$ omitted)}

\begin{tabular}{|c|c|c|c|}
	\hline
	$\Gamma$ & \thead{minimal \\ symmetrizer $D$} & $\Gamma'$ &  \thead{ $(\Gamma', \sigma)$ in \cite{Lus}} \\
	\hline
	\multirow{2}{*}{\thead{	\begin{tikzpicture}
			\draw (-0.1,0) node[anchor=east]  {$\widetilde{B}_n$ : };
			\draw (0,0) -- (0.6,0)-- (1.2,0) (2.4,0)-- (3,0)-- (3.6,0);
			\draw[fill] (1.6,0) circle [radius=0.025];
			\draw[fill] (1.8,0) circle [radius=0.025];
			\draw[fill] (2.0,0) circle [radius=0.025];
			\draw[fill] (0,0) circle [radius=0.05];
			\draw[fill] (0.6,0) circle [radius=0.05];
			\draw[fill] (1.2,0) circle [radius=0.05];
			\draw[fill] (2.4,0) circle [radius=0.05];
			\draw[fill] (3,0) circle [radius=0.05];
			\draw[fill] (3.6,0) circle [radius=0.05];
			\node [below] at (1.5,-0.4) {($n\geq 2$)};
			\node [below] at (0.3,0.4) {\tiny{(1,2)}};
			\node [below] at (3.3,0.4) {\tiny{(2,1)}};
			\node [below] at (0,0) {1};
			\node [below] at (0.6,0) {2};
			\node [below] at (1.2,0) {3};
			\node [below] at (2.4,0) {n-1};
			\node [below] at (3.6,0) {n+1};
			\node [below] at (3,-0.06) {n};
			\end{tikzpicture}}} &  	\multirow{2}{*}{\thead{${\rm diag}(2,1,\cdots,1,2)$}} &  If $p=2$, $\widetilde{B}_n$ & \\ \cline{3-4} & &  If $p\neq 2$, $\widetilde{D}_{n+2}$  & \cite[14.1.5 (c)]{Lus}\\
	\hline
	\multirow{2}{*}{\thead{	\begin{tikzpicture}
			\draw (-0.1,0) node[anchor=east]  {$\widetilde{C}_n$ : };
			\draw (0,0)-- (0.6,0)-- (1.2,0) (2.4,0)-- (3,0)-- (3.6,0);
			\draw[fill] (1.6,0) circle [radius=0.025];
			\draw[fill] (1.8,0) circle [radius=0.025];
			\draw[fill] (2.0,0) circle [radius=0.025];
			\draw[fill] (0,0) circle [radius=0.05];
			\draw[fill] (0.6,0) circle [radius=0.05];
			\draw[fill] (1.2,0) circle [radius=0.05];
			\draw[fill] (2.4,0) circle [radius=0.05];
			\draw[fill] (3,0) circle [radius=0.05];
			\draw[fill] (3.6,0) circle [radius=0.05];
			\node [below] at (1.5,-0.4) {($n\geq 2$)};
			\node [below] at (0.3,0.4) {\tiny{(2,1)}};
			\node [below] at (3.3,0.4) {\tiny{(1,2)}};
			\node [below] at (0,0) {1};
			\node [below] at (0.6,0) {2};
			\node [below] at (1.2,0) {3};
			\node [below] at (2.4,0) {n-1};
			\node [below] at (3.6,0) {n+1};
			\node [below] at (3,-0.06) {n};
			\end{tikzpicture}}} &  	\multirow{2}{*}{\thead{${\rm diag}(1,2,\cdots,2,1)$}} &  If $p=2$, $\widetilde{C}_n$ & \\ \cline{3-4} & &\   If $p\neq 2$, $\widetilde{A}_{2n-1}$  & \cite[14.1.5 (a)]{Lus}\\
	\hline
	\multirow{2}{*}{\thead{	\begin{tikzpicture}
			\draw (-0.1,0) node[anchor=east]  {$\widetilde{A}_{11}$ : };
			\draw (0,0) -- (0.6,0);
			\draw[fill] (0,0) circle [radius=0.05];
			\draw[fill] (0.6,0) circle [radius=0.05];
			\node [below] at (0.3,0.4) {\tiny{(1,4)}};
			\node [below] at (0,0) {1};
			\node [below] at (0.6,0) {2};
			\end{tikzpicture}}} & \multirow{2}{*}{\thead{${\rm diag}(4,1)$}}  &  If $p=2$, $\widetilde{A}_{11}$ & \\ \cline{3-4} & &  If $p\neq 2$, $\widetilde{D}_4$  & \cite[14.1.5 (e)]{Lus}\\
	\hline
	\multirow{2}{*}{\thead{\begin{tikzpicture}
			\draw (-0.1,0) node[anchor=east]  {$\widetilde{BC}_n$ : };
			\draw (0,0) -- (0.6,0)-- (1.2,0) (2.4,0)-- (3,0)-- (3.6,0);
			\draw[fill] (1.6,0) circle [radius=0.025];
			\draw[fill] (1.8,0) circle [radius=0.025];
			\draw[fill] (2.0,0) circle [radius=0.025];
			\draw[fill] (0,0) circle [radius=0.05];
			\draw[fill] (0.6,0) circle [radius=0.05];
			\draw[fill] (1.2,0) circle [radius=0.05];
			\draw[fill] (2.4,0) circle [radius=0.05];
			\draw[fill] (3,0) circle [radius=0.05];
			\draw[fill] (3.6,0) circle [radius=0.05];
			\node [below] at (1.5,-0.4) {($n\geq 2$)};
			\node [below] at (0.3,0.4) {\tiny{(1,2)}};
			\node [below] at (3.3,0.4) {\tiny{(1,2)}};
			\node [below] at (0,0) {1};
			\node [below] at (0.6,0) {2};
			\node [below] at (1.2,0) {3};
			\node [below] at (2.4,0) {n-1};
			\node [below] at (3.6,0) {n+1};
			\node [below] at (3,-0.06) {n};
			\end{tikzpicture}}} &  	\multirow{2}{*}{\thead{${\rm diag}(4,2,\cdots,2,1)$}} &  If $p=2$, $\widetilde{BC}_n$ &  \\ \cline{3-4} & &\   If $p\neq 2$, $\widetilde{D}_{2n+2}$   & \cite[14.1.5 (e)]{Lus}\\
	\hline
	\multirow{2}{*}{\thead{\begin{tikzpicture}
			\draw (-0.1,0) node[anchor=east]  {$\widetilde{BD}_n$ : };
			\draw (0,0) -- (0.6,0)-- (1.2,0) (2.4,0)-- (3,0)-- (3.6,0.3);
			\draw (3,0)-- (3.6,-0.3);
			\draw[fill] (1.6,0) circle [radius=0.025];
			\draw[fill] (1.8,0) circle [radius=0.025];
			\draw[fill] (2.0,0) circle [radius=0.025];
			\draw[fill] (0,0) circle [radius=0.05];
			\draw[fill] (0.6,0) circle [radius=0.05];
			\draw[fill] (1.2,0) circle [radius=0.05];
			\draw[fill] (2.4,0) circle [radius=0.05];
			\draw[fill] (3,0) circle [radius=0.05];
			\draw[fill] (3.6,0.3) circle [radius=0.05];
			\draw[fill] (3.6,-0.3) circle [radius=0.05];
			\node [below] at (1.5,-0.4) {($n\geq 3$)};
			\node [below] at (0.3,0.4) {\tiny{(1,2)}};
			\node [below] at (0,0) {1};
			\node [below] at (0.6,0) {2};
			\node [below] at (1.2,0) {3};
			\node [below] at (3,0) {n-1};
			\node [above] at (3.6,0.3) {n+1};
			\node [below] at (3.6,-0.3) {n};
			\end{tikzpicture}}} &  	\multirow{2}{*}{\thead{${\rm diag}(2,1,\cdots,1)$}} &  If $p=2$, $\widetilde{BD}_n$ & \\ \cline{3-4} & &   If $p\neq 2$, $\widetilde{D}_{n+1}$  & \cite[14.1.5 (b)]{Lus}  \\
	\hline
	\multirow{2}{*}{\thead{	\begin{tikzpicture}
			\draw (-0.1,0) node[anchor=east]  {$\widetilde{CD}_n$ : };
			\draw (0,0) -- (0.6,0)-- (1.2,0) (2.4,0)-- (3,0)-- (3.6,0.3);
			\draw (3,0)-- (3.6,-0.3);
			\draw[fill] (1.6,0) circle [radius=0.025];
			\draw[fill] (1.8,0) circle [radius=0.025];
			\draw[fill] (2.0,0) circle [radius=0.025];
			\draw[fill] (0,0) circle [radius=0.05];
			\draw[fill] (0.6,0) circle [radius=0.05];
			\draw[fill] (1.2,0) circle [radius=0.05];
			\draw[fill] (2.4,0) circle [radius=0.05];
			\draw[fill] (3,0) circle [radius=0.05];
			\draw[fill] (3.6,0.3) circle [radius=0.05];
			\draw[fill] (3.6,-0.3) circle [radius=0.05];
			\node [below] at (1.5,-0.4) {($n\geq 3$)};
			\node [below] at (0.3,0.4) {\tiny{(2,1)}};
			\node [below] at (0,0) {1};
			\node [below] at (0.6,0) {2};
			\node [below] at (1.2,0) {3};
			\node [below] at (3,0) {n-1};
			\node [above] at (3.6,0.3) {n+1};
			\node [below] at (3.6,-0.3) {n};
			\end{tikzpicture}}} &  	\multirow{2}{*}{\thead{${\rm diag}(1,2,\cdots,2)$}} &  If $p=2$, $\widetilde{CD}_n$ & \\ \cline{3-4} & &   If $p\neq 2$, $\widetilde{D}_{2n}$   & \cite[14.1.5 (d)]{Lus} \\
	\hline
	\multirow{2}{*}{\thead{\begin{tikzpicture}
			\draw (-0.1,0) node[anchor=east]  {$\widetilde{F}_{41}$ : };
			\draw (0,0) -- (0.6,0)-- (1.2,0)--(1.8,0)-- (2.4,0);
			\draw[fill] (0,0) circle [radius=0.05];
			\draw[fill] (0.6,0) circle [radius=0.05];
			\draw[fill] (1.2,0) circle [radius=0.05];
			\draw[fill] (2.4,0) circle [radius=0.05];
			\draw[fill] (1.8,0) circle [radius=0.05];
			\node [below] at (1.5,0.4) {\tiny{(1,2)}};
			\node [below] at (0,0) {1};
			\node [below] at (0.6,0) {2};
			\node [below] at (1.2,0) {3};
			\node [below] at (1.8,0) {4};
			\node [below] at (2.4,0) {5};
			\end{tikzpicture}}} &
	\multirow{2}{*}{\thead{${\rm diag}(2,2,2,1,1)$ }}&
	If $p=2$, $\widetilde{F}_{41}$ & \\ \cline{3-4} & &  If $p\neq 2$, $\widetilde{E}_7$ & \cite[14.1.5 (i)]{Lus} \\
	\hline
	\multirow{2}{*}{\thead{	\begin{tikzpicture}
			\draw (-0.1,0) node[anchor=east]  {$\widetilde{F}_{42}$ : };
			\draw (0,0) -- (0.6,0)-- (1.2,0)--(1.8,0)-- (2.4,0);
			\draw[fill] (0,0) circle [radius=0.05];
			\draw[fill] (0.6,0) circle [radius=0.05];
			\draw[fill] (1.2,0) circle [radius=0.05];
			\draw[fill] (2.4,0) circle [radius=0.05];
			\draw[fill] (1.8,0) circle [radius=0.05];
			\node [below] at (1.5,0.4) {\tiny{(2,1)}};
			\node [below] at (0,0) {1};
			\node [below] at (0.6,0) {2};
			\node [below] at (1.2,0) {3};
			\node [below] at (1.8,0) {4};
			\node [below] at (2.4,0) {5};
			\end{tikzpicture}}} &
	\multirow{2}{*}{\thead{${\rm diag}(1,1,1,2,2)$ }}&
	If $p=2$, $\widetilde{F}_{42}$ & \\ \cline{3-4} & &  If $p\neq 2$, $\widetilde{E}_6$  & \cite[14.1.5 (g)]{Lus} \\
	\hline
	\multirow{2}{*}{\thead{	\begin{tikzpicture}
			\draw (-0.1,0) node[anchor=east]  {$\widetilde{G}_{21}$ : };
			\draw (0,0) -- (0.6,0)-- (1.2,0);
			\draw[fill] (0,0) circle [radius=0.05];
			\draw[fill] (0.6,0) circle [radius=0.05];
			\draw[fill] (1.2,0) circle [radius=0.05];
			\node [below] at (0.9,0.4) {\tiny{(1,3)}};
			\node [below] at (0,0) {1};
			\node [below] at (0.6,0) {2};
			\node [below] at (1.2,0) {3};
			\end{tikzpicture}}} & \multirow{2}{*}{\thead{${\rm diag}(3,3,1)$}}  &  If $p=3$, $\widetilde{G}_{21}$  & \\ \cline{3-4} & &  If $p\neq 3$, $\widetilde{E}_6$  & \cite[14.1.5 (h)]{Lus}\\
	\hline
	\multirow{2}{*}{\thead{	\begin{tikzpicture}
			\draw (-0.1,0) node[anchor=east]  {$\widetilde{G}_{22}$ : };
			\draw (0,0) -- (0.6,0)-- (1.2,0);
			\draw[fill] (0,0) circle [radius=0.05];
			\draw[fill] (0.6,0) circle [radius=0.05];
			\draw[fill] (1.2,0) circle [radius=0.05];
			\node [below] at (0.9,0.4) {\tiny{(3,1)}};
			\node [below] at (0,0) {1};
			\node [below] at (0.6,0) {2};
			\node [below] at (1.2,0) {3};
			\end{tikzpicture}}} & \multirow{2}{*}{\thead{${\rm diag}(1,1,3)$}}  &  If $p=3$, $\widetilde{G}_{22}$ & \\ \cline{3-4} & &  If $p\neq 3$, $\widetilde{D}_4$ & \cite[14.1.5 (f)]{Lus}  \\
	\hline
\end{tabular}
\end{table}

\subsection{Comparison with the correspondence in \cite[Section 14.1]{Lus}}

In this subsection, we compare the construction in Subsection \ref{sub:4.2} to the well-known correspondence in \cite[Section 14.1]{Lus}, which is between Cartan matrices and graphs with admissible automorphisms.

Denote by $\Gamma$ a finite graph without loops. An automorphism $\sigma\colon \Gamma\rightarrow \Gamma$ is \emph{admissible}, provided that there is no edge joining two vertices in the same $\sigma$-orbit. Since we will not distinguish parallel edges, two automorphisms are identified if their actions on vertices are the same. Similarly, we have the notion of \emph{admissible automorphisms} for finite acyclic quivers.

Let $C=(c_{ij})\in M_n(\mathbb{Z})$ be a Cartan matrix with $D={\rm diag}(c_1, \cdots, c_n)$ its symmetrizer. We associate a graph $\Gamma$ with an admissible automorphism $\sigma$ as follows. The vertex set of $\Gamma$ is given by $\Gamma_0=\{(i, l_i)\; |\; 1\leq i\leq n, 0\leq l_i<c_i\}$. There is an edge between $(i, l_i)$ and $(j, l_j)$ if and only if $i\neq j$, $c_{ij}\neq 0$ and $l_i\equiv l_j ({\rm mod}\; {\rm gcd}(c_i, c_j))$, in which case there are exactly ${\rm gcd}(c_{ij}, c_{ji})$ such edges. The automorphism $\sigma$ sends $(i, l_i)$ to $(i, l_i+1)$, where we identify $(i, 0)$ with $(i, c_i)$.

The following well-known correspondence between Cartan matrices and graphs with admissible automorphisms is due to \cite[Section 14.1]{Lus}. We use slightly different terminologies.

\begin{prop}
The above assignment yields a bijection between the corresponding sets of isomorphism classes
\[
\left\{(C, D) \; \Bigg |\; \begin{aligned}  &C \mbox{ a Cartan matrix}, \\
                                     & D \mbox{ its symmetrizer} \end{aligned} \right\}\longleftrightarrow \left\{ (\Gamma, \sigma) \; \Bigg |\; \begin{aligned} &\Gamma \mbox{ a finite graph}, \sigma \mbox{ an  admissible}\\
                                                                       &\mbox{automorphism of } \Gamma\end{aligned} \right\}.
                                     \]
\end{prop}

\begin{proof}
The inverse map is given as follows. We assume that there are exactly $n$ $\sigma$-orbits in $\Gamma_0$. We will index the rows and columns of the matrices by the orbit set $\Gamma_0/\sigma$. For a $\sigma$-orbit $[i]$, we set $c_{[i]}$ to be its cardinality. This defines the diagonal matrix $D$. For $[i]\neq [j]$, we define $c_{[i], [j]}=-\frac{N}{c_{[i]}}$, where $N$ is the total number of edges between $i'$ and $j'$ with all possible $i'\in [i]$ and $j'\in [j]$. This defines the Cartan matrix $C$.
\end{proof}

We now add the orientations into consideration. Then we have the following immediate consequence. We emphasize that two automorphisms of a quiver are identified, provided that their actions on vertices are the same; compare \cite[the second paragraph of Introduction]{Hu04}.

\begin{cor}\label{cor:quiver}
There is a bijection between the corresponding sets of isomorphism classes
\[
\left\{(C, D, \Omega) \; |\; \mbox{a Cartan tripe} \right\}\longleftrightarrow \left\{ (\Delta, \sigma) \; \Bigg |\; \begin{aligned} &\Delta \mbox{ a finite acyclic quiver},\\
                                                                       &  \sigma \mbox{ an  admissible automorphism of } \Delta\end{aligned} \right\}.
                                     \]
\end{cor}

Let $(C, D, \Omega)$ be a Cartan triple and $\mathcal{C}=\mathcal{C}(C, D, \Omega)$ be its free EI category. Let $k$ be a field with characteristic $p$. Then the category algebra $k\mathcal{C}$ is hereditary if and only if $p$ is zero or coprime to each entry of $D$; see \cite[Theorem 1.2]{L2011}.

We have the following immediate consequence of  Theorem \ref{thm:main}. Roughly speaking,  the algebra isomorphism  in the hereditary case yields an algebraic enrichment of the correspondence in Corollary \ref{cor:quiver}. For a different enrichment, we refer to the isomorphism in \cite[Lemma 21]{Hu04}, which is  over finite fields.

\begin{prop}\label{prop:compa}
Let $\mathcal{C}=\mathcal{C}(C, D, \Omega)$ be as above, and let $p$, the characteristic of $k$, be zero or coprime to each entry of $D$. Assume that $k$ has enough roots of unity. Then there is an isomorphism of algebras
$$k\mathcal{C}\stackrel{\sim}\longrightarrow k\Delta,$$
where $\Delta$ is the acyclic quiver corresponding to $(C, D, \Omega)$ in Corollary \ref{cor:quiver}.
\end{prop}

\begin{proof}
Recall the construction of $(C', D', \Omega')$ in Subsection \ref{sub:4.2}. It follows that $D'$ is the identity matrix and $C'$ is symmetric. In particular, the algebra $H(C', D', \Omega')$ is isomorphic to the path algebra $k\Delta$ for some acyclic quiver $\Delta$; compare Remark~\ref{rem:GLS}. Indeed, the quiver $\Delta$ is exactly the one in Corollary \ref{cor:quiver}, which corresponds to $(C, D, \Omega)$. Then we are done by Theorem \ref{thm:main}.
\end{proof}

\begin{rem}
It is not clear how the admissible automorphism $\sigma$ of $\Delta$ interacts with the above algebra isomorphism. In particular, it would be nice to link this isomorphism  to the representation theory of quivers with automorphisms; see \cite{Hu04}.
\end{rem}

\vskip 10pt

\noindent {\bf Acknowledgements.}\quad  The paper is completed when the second named author is visiting University of Stuttgart; she thanks Steffen Koenig and Julian Kuelshammer for many helpful suggestions.  This work is supported by the National Natural Science Foundation of China (No.s 11522113, 11571329, 11671174 and 11671245), and the Fundamental Research Funds for the Central Universities.

\bibliography{}

\vskip 10pt

 {\footnotesize \noindent Xiao-Wu Chen, Ren Wang\\
 Key Laboratory of Wu Wen-Tsun Mathematics, Chinese Academy of Sciences,\\
 School of Mathematical Sciences, University of Science and Technology of China, Hefei 230026, Anhui, PR China}

\end{document}